\documentclass[11pt]{article}
\usepackage{graphicx}
\usepackage{epstopdf}
\usepackage{amsmath, amssymb}
\usepackage{latexsym, color}
\usepackage{mathtools}
\usepackage{xcolor}
\def\la{\big\langle}
\def\ra{\big\rangle}

\def\ds{\displaystyle}
\def\forall{\hbox{for all}~}
\def\L{{\bf L}}

\def\bfn{{\bf n}}
\def\bfe{{\bf e}}

\def\ve{\varepsilon}

\def\E{{\cal E}}
\def\A{{\cal A}}

\def\H{{\cal H}}\def\caL{{\cal L}}

\def\R{I\!\!R}

\def\implies{\Longrightarrow}
\def\vp{\varphi}

\def\F{{\cal F}}

\def\v{\vskip 1em}
\def\O{{\cal O}}

\def\C{{\cal C}}

\def\H{{\cal H}}
\def\T{{\cal T}}

\def\bega{\begin{array}}
\def\enda{\end{array}}
\def\begi{\begin{itemize}}
\def\endi{\end{itemize}}
\def\ov{\overline}

\def\Tilde{\widetilde}
\def\Hat{\widehat}

\def\meas{\hbox{meas}}
\def\bel{\begin{equation}\label}
\def\eeq{\end{equation}}
\def\sqr#1#2{\vbox{\hrule height .#2pt
\hbox{\vrule width .#2pt height #1pt \kern #1pt
\vrule width .#2pt}\hrule height .#2pt }}
\def\square{\sqr74}
\def\endproof{\hphantom{MM}\hfill\llap{$\square$}\goodbreak}
\definecolor{cadmiumgreen}{rgb}{0.0, 0.42, 0.24}

\newtheorem{theorem}{Theorem}[section]

\newtheorem{lemma}{Lemma}[section]
\newtheorem{proposition}{Proposition}[section]
\newtheorem{remark}{Remark}[section]
\newtheorem{definition}{Definition}[section]
\newtheorem{example}{Example}[section]
\newtheorem{conjecture}{Conjecture}[section]

\begin{document}

\title{\bf On the Value Function for a Class of Set Motion Problems}

\author{Alberto Bressan$^{(1)}$, Maria Teresa Chiri$^{(2)}$ and Elsa M. Marchini$^{(3)}$\\
\, \\
{\small $^{(1)}$Department of Mathematics, Penn State University, University Park, PA~16802, USA.}\\
{\small $^{(2)}$~Department of Mathematics and  Statistics, Queen's University,
Kingston, ON K7L3N6,
Canada.}\\
{\small $^{(3)}$Dipartimento di Matematica, Politecnico di Milano,
Piazza L.\,da Vinci, 32 - 20133 Milano, Italy.}\\
\, \\
{\small E-mails: axb62@psu.edu,~maria.chiri@queensu.ca, ~elsa.marchini@polimi.it.}
}
\maketitle

\begin{abstract} The paper considers a class of optimal set motion problems, where the control
determines the velocity of boundary points in the normal direction.
Results and counterexamples are given, on the continuous dependence of the minimum cost
from the initial configuration.
The last sections of the paper study the slicing cost of bounded open sets.
For strongly Lipschitz
domains, it is proved that the slicing cost varies continuously w.r.t.~the Hausdorff distance
of the boundaries.    We conjecture that the slicing cost for a disc is maximal among all sets with the same area.
\end{abstract}

\v

\section{Introduction}
\label{s:1}
\setcounter{equation}{0}
We consider a family of geometric evolution problems, modeling the spatial control 
of an invasive population 
\cite{BCS1}.
For each time $t\in [0,T]$, we denote by $\Omega(t)\subset\R^2$ a set moving in the plane.
This can be regarded as a ``contaminated region", which we would like to shrink as 
much as possible.   To control the evolution of this set, we assign the 
velocity $\beta=\beta(t,x)$ in the inward normal direction at every boundary point $x\in \partial \Omega(t)$.

A function $E(\beta)\geq 0$ is given, 
describing the {\bf effort} needed to push the boundary of $\Omega(t)$ inward, with speed $\beta$ in the 
normal direction (see Fig.~\ref{f:sm345}). 
 The {\bf total control effort} at time $t\in [0,T]$
is then defined as
\bel{Et} \E(t)~\doteq~
\int_{\partial\Omega(t)}E\bigl(\beta(t,x)\bigr)\, \H^1(dx),
\eeq
where the integral is computed w.r.t.~the 1-dimensional Hausdorff measure
along the boundary of $\Omega(t)$. 
In this paper we focus on the case where
\bel{E} E(\beta)~=~\max\bigl\{ 0, 1+\beta\bigr\}.\eeq
Given a constant $M>0$ accounting for the maximum control effort, we consider set motions
$t\mapsto \Omega(t)$
which satisfy the constraint
\bel{EM} \E(t)~\leq~M\qquad\forall ~t\in [0,T].\eeq

\begin{figure}[ht]
\centerline{\hbox{\includegraphics[width=6cm]{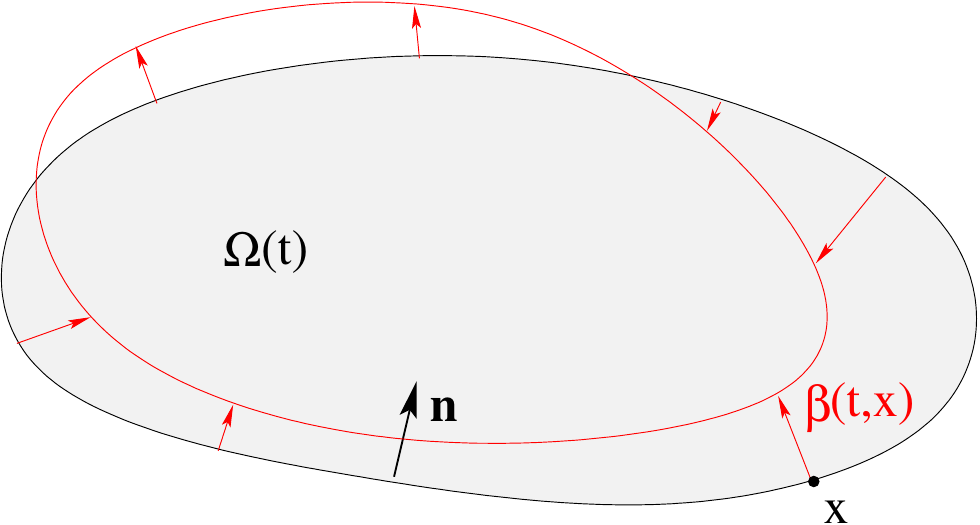}}
\qquad\qquad \hbox{\includegraphics[width=5cm]{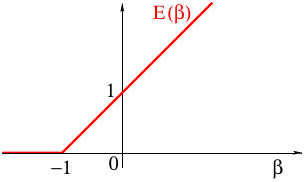}}}
\caption{\small Left: a moving set, where the evolution is determined by assigning the inward normal speed $\beta$
at each boundary point.
Right: the effort function $E(\beta)$ in (\ref{E}).
}
\label{f:sm345}
\end{figure}

This models a situation where:
\begi
\item If the control effort is everywhere zero: $E(\beta)=0$, then the inward normal speed is $\beta= -1$
at every point. 
Hence the contaminated set $\Omega(t)$ expands with unit speed in all directions.
In particular, its area increases at a rate equal to the perimeter:
$${d\over dt} \caL^2\bigl(\Omega(t)\bigr)~=~- \int_{\partial \Omega(t)} \beta(t,x) \, \H^1(dx)~=~ 
\H^1\bigl(\partial \Omega(t)\bigr).$$
\item By implementing a control with total effort  $\E(t)= M$, we can clean up a region of area $M$ per unit time:
$$\bega{l}\ds {d\over dt} \caL^2\bigl(\Omega(t)\bigr)~=~-\int_{\partial \Omega(t)}  \beta(t,x) \, \H^1(dx)\\[4mm]
\ds\quad =~
\int_{\partial \Omega(t)} \Big[1 - E\bigl(\beta(t,x)\bigr)\Big] \, \H^1(dx)~=~\H^1\bigl(\partial \Omega(t)\bigr)-M.\enda$$
\endi
Here and in the sequel, $\H^1$ denotes the 1-dimensional Hausdorff measure while  $ \caL^2$ is the 2-dimensional 
Lebesgue measure. 

We say that the set motion $t\mapsto\Omega(t)$ is {\bf admissible} if 
it satisfies the constraint (\ref{EM}).  

%
%
%
%
%

In the above setting, two basic problems can be formulated.

\begi
\item[{\bf (EP)}] {\bf Eradication Problem.}  
{\it   Let an initial set $\ov\Omega\subset\R^2$ and a constant $M>0$
be given. Find an admissible set-valued function $t\mapsto \Omega(t)$  such that, for some $T>0$,
\bel{nco}\Omega(0)~=~\ov\Omega,\qquad\qquad 
\Omega(T)~=~\emptyset,\eeq
}
\endi
\v
\begi
\item[{\bf (MTP)}]  {\bf Minimum Time Problem.}  {\it  Among all admissible strategies that satisfy
(\ref{nco}), find one which minimizes the time $T$.}
\endi

Notice that the constant $M$ puts an upper bound on the total control effort at each time $t$.
For example, in a pest eradication problem, there will be an upper bound
on the amount of pesticides that can be sprayed per unit time.  
If $M$ is too small, compared with the size of the contaminated region, 
it may not be possible to completely eradicate the invasive population.  

More generally, even if the Eradication Problem does not have a solution,
one can consider an optimization problem, minimizing the size of the contaminated set over time.
\v
\begi
\item[{\bf (OP)}]  {\bf Optimization Problem.}  {\it  Given an initial set $\ov\Omega$ and a constants $\kappa >0$, find an admissible motion 
$t\mapsto \Omega(t)$ 
that minimizes the cost functional 
$$J(\Omega)~\doteq~\kappa\int_\tau^T \caL^2\bigl(\Omega(t)\bigr)\, dt +
\caL^2\bigl(\Omega(T)\bigr)$$
with initial data
$ \Omega(\tau)~=~\ov\Omega$.
}
\endi

The main results in \cite{BCS2} provide the existence of optimal solutions,
together with necessary conditions for optimality.  
In the case where the initial set $\ov\Omega\subset\R^2$ is convex, the optimal solution
to  {\bf (MTP)} and {\bf (OP)}  has been explicitly determined in 
\cite{BBC}.

In the present paper we focus on the value functions for the optimization problems 
{\bf (MTP)} and {\bf (OP)}.  More precisely

\begin{definition}
The {\bf minimum time function} for the problem {\bf (MTP)}, denoted by $\T(\ov\Omega)$, is the infimum among all times $T>0$ such that there exists an admissible motion $t\mapsto \Omega(t)$
 with $\Omega(0)=\ov\Omega$ and $\Omega(T)=\emptyset$.
\end{definition}

\begin{definition} Given $\kappa>0$,
the {\bf value function} corresponding to the optimization problem {\bf (OP)}, denoted by 
$V(\tau, \ov\Omega)$ is defined as
\bel{Vto} V(\tau,\ov\Omega)~\doteq~\inf_{\Omega\in \A}
\left\{ \kappa \int_\tau^T \caL^2\bigl(\Omega(t)\bigr)\, dt +
\caL^2\bigl(\Omega(T)\bigr)\right\}.\eeq
Here the infimum is sought among all admissible motions $t\mapsto \Omega(t)$, with $\Omega(\tau)=\ov\Omega$.
\end{definition}

A positive result on the continuity of these functions, and a counterexample, will be given in 
Section~\ref{sec:2} and \ref{sec:3}, respectively.

We shall also consider similar control problems for moving sets, but with 
geographical constraints \cite{BMS}.  These models describe an invasive biological population within an island,
where the sea provides a natural barrier to its expansion.
More precisely,
given an open set $G\subset\R^2$
we impose the additional constraint
$\Omega(t)\subseteq G$.
Since the population cannot propagate outside $G$,
 the instantaneous control effort (\ref{Et}) is now replaced by
\bel{CEt} \E(t)~\doteq~
\int_{\partial\Omega(t)\cap G}E\bigl(\beta(t,x)\bigr)\, \H^1(dx).
\eeq
Notice that in (\ref{CEt}) the effort is integrated only over the relative boundary of $\Omega(t)$,
contained inside the open set $G$.

The three problems {\bf (EP)}, {\bf (MTP)} and {\bf (OP)}  can now be formulated 
in the same way as before, replacing (\ref{Et}) with (\ref{CEt}).

Assuming that the initial contamination is spread over the entire island: $\Omega(0)=G$, 
we study the continuity of the minimum time function $\T(G)$ as the set $G\subset\R^2$ varies.
\v
The second part of the paper is devoted to the value function for a  family of geometric problems.
As shown in \cite{BCM26}, these can be obtained as a limiting case of the eradication problems,
for suitable values of the parameters.

\begin{definition} Let $G\subset\R^2$  be a bounded open 
set with finite perimeter, and call  $T=\caL^2(G)$ its area.
By a {\bf slicing} of $G$ we mean a set-valued map $t\mapsto\Omega(t)\subseteq G$, $t\in [0,T]$, with the following properties:
\bel{Om1}
0\leq t_1 < t_2\leq T\qquad\implies\qquad \Omega(t_1)\subset\Omega(t_2),\eeq
\bel{Om2} \caL^2\bigl( \Omega(t)\bigr)~=~t\qquad\quad\forall t\in [0,T].\eeq
\end{definition}
Among all slicings of the set $G$, we seek one that minimizes the average length of the 
relative boundaries. This leads to
\begi
\item[{\bf (OSP)}] {\bf (Optimal Slicing Problem).} {\it Let $G\subset\R^2$  be a bounded open 
set with finite perimeter and area $T = \caL^2(G)$.
Among all slicings $t\mapsto\Omega(t)$ of the set $G$, minimize the integral
$$J(\Omega)~=~\int_0^T \H^1\bigl( G\cap\partial \Omega(t)\bigr)\, dt.$$
}\endi
Sweepout problems of this type, where a one-parameter family of curves or regions is required to traverse a surface while minimizing a length-type cost, have also been studied in connection with widths and min-max constructions; see e.g.\ \cite{CL19, L14}.

Calling $V(G)$ the infimum slicing cost for the set $G$, in Section~\ref{sec:4} we study the continuous
dependence of this cost  on the set $G$.

We conjecture that, among all sets of the same area, the disc has the 
largest slicing cost.   Partial results in this direction are discussed  in Section~\ref{sec:5}.
An explicit computation of the slicing cost for a disc is given in the Appendix.

\section{Continuity of the value function}
\label{sec:2}
\setcounter{equation}{0}

In the regular case where the boundaries $\partial\Omega(t)$ admit a $\C^1$ parameterization,
the definitions of interior normal vector and of normal velocity $\beta$ in the inward direction are
clear.   However, in general the existence of optimal set motions can be achieved only within a family of sets with finite perimeter. In such case the normal vector is well defined only a.e.~w.r.t.~the Hausdorff measure.  As a preliminary, we recall here the precise definition of admissible motion, as introduced in \cite{BCS2}.

 Call $\F$ the family of all sets
$\Omega\subset [0,T]\times\R^2$ with finite perimeter.  Each $\Omega\in \F$ determines a 
set-motion
\bel{12}t~\mapsto~\Omega(t)~\doteq~\bigl\{ x\,;~~(t,x)\in \Omega\bigr\}.\eeq
For every point $(t,x)\in \partial^\star \Omega$ in the reduced boundary of $\Omega$ (see \cite{AFP, M} for 
a precise definition), 
let $\nu(t,x)= (\nu_0, \nu_1, \nu_2)\in \R^3$ be the (inward pointing) unit normal vector.
%

The (inward) normal velocity of the set 
$\Omega(t)$  at the point 
$(t,x)\in \partial^\star\Omega$ is then computed by
\bel{bdef}\beta~=~ {-\nu_0\over \sqrt{\nu_1^2 + \nu_2^2}}\,.\eeq
Recalling that $E(\beta)\doteq\max\{ 1+\beta, 0\}$,  the instantaneous effort is thus computed by
$$\E(t)
~=~\int_{\partial\Omega(t)} E\left({-\nu_0\over \sqrt{\nu_1^2 + \nu_2^2}}\right)
\,\H^1(dx)
~=~\int_{\partial\Omega(t)} \max\left\{ {-\nu_0+ \sqrt{\nu_1^2 + \nu_2^2}\over  \sqrt{\nu_1^2 + \nu_2^2}}\,,~0\right\}
\,\H^1(dx).$$
Therefore, integrating over a time interval $t\in \,]t_1,t_2[$ one finds
$$\int_{t_1}^{t_2} \E(t)\, dt ~=~\int_{\partial^*\Omega\cap\{(t,x);\, t_1<t<t_2\}} 
\max\Big\{ -\nu_0 +\sqrt{\nu_1^2+\nu_2^2}\,,~0\Big\}\, d\H^2\,.$$

To  impose the requirement that $\E(t)\leq M$ for all $t$,    we consider the convex, positively homogeneous
function $L:\R^3\mapsto \R$, defined as
$$L(v) ~=~L(v_0,v_1,v_2)~\doteq~\max\Big\{ -v_0 +\sqrt{v_1^2+v_2^2}\,,~0\Big\}.$$\begin{definition}\label{d:13}
Given $M>0$, we say that  a set with finite perimeter $\Omega\in\F$ represents an  
{\bf admissible motion}, as in (\ref{12}), and write
$\Omega\in \A$, if for every $0\leq t_1<t_2\leq T$ one has
$$\int_{\partial^*\Omega\cap\{(t,x);\, t_1<t<t_2\}} L\bigl(\nu(t,x)\bigr)\, d\H^2~\leq~M(t_2-t_1).$$
\end{definition}

Next, we look at the value function for the problem {\bf (OP)}.

It is a trivial observation that 
$$\ov\Omega_1\subseteq \ov\Omega_2\qquad\implies\qquad V(\tau,\ov\Omega_1)\leq V(\tau, \ov\Omega_2).$$
In the remainder of this section we study the continuity properties of
the map $\ov\Omega\mapsto V(\ov\Omega)$.

For a given $\kappa>0$, we denote by
$\F_\kappa$ the family of all bounded sets $\Omega\subset\R^2$ whose perimeter
satisfies $$\H^1(\partial \Omega)~\leq ~\kappa.$$
A distance on $\F_\kappa$ can be defined in terms of the measure of the symmetric difference:
\bel{DFK}
d(\Omega, \Omega')~\doteq~\caL^2(\Omega~\Delta~\Omega')~=~\big\|
{\bf 1}_{\Omega}-{\bf 1}_{\Omega'}\bigr\|_{\L^1}\,.\eeq
Here and in the sequel, ${\bf 1}_A$ denotes the characteristic function of a set $A$.

\begin{proposition}\label{p:21} Assume $M> 2\kappa$.   Then, for any $T>0$,  the value function 
$V(\tau,\ov\Omega)$ at (\ref{Vto}) is Lipschitz continuous restricted to $[0,T]\times \F_\kappa$.
\end{proposition}

{\bf Proof.} {\bf 1.} To prove Lipschitz continuity w.r.t.~the time $\tau$, assume 

Let $t\mapsto \Omega(t)$, $t\in [\tau,T]$ be an optimal strategy achieving the 
minimum cost $V(\tau,\ov\Omega)$.   Then the strategy
$$t~\mapsto~\Omega(t-(\tau'-\tau)\bigr)$$
achieves a cost 
\bel{com1} J(\Omega)~\leq~ V(\tau,\ov\Omega)+ M(\tau'-\tau).\eeq

On the other hand, let $t\mapsto \Omega'(t)$, $t\in [\tau',T]$  be an optimal strategy 
achieving the minimum cost $V(\tau',\ov\Omega)$.    
Consider the strategy
$$\Tilde\Omega(t)~=~\left\{\bega{cl} \ov\Omega\quad&\hbox{if}~~~t\in [\tau, \tau'],\\[1mm]
\Omega'(t)\quad &\hbox{if}~~~t\in [\tau',T].\enda\right.$$
If $\ov\Omega\in \F_{\kappa}$ and $M\geq \kappa$, this    strategy is admissible. Its
cost is
\bel{com2}J(\Tilde\Omega) ~=~\kappa (\tau'-\tau)\caL^2(\ov\Omega)+ V(\tau',\ov\Omega).\eeq
Together, (\ref{com1})-(\ref{com2}) yield Lipschitz continuity of the value function w.r.t.~time.
\v
{\bf 2.}
To prove Lipschitz continuity w.r.t.~the initial data, assume
$d(\ov\Omega_1, \ov\Omega_2)=\delta$.   
By monotonicity, it suffices to estimate the difference
$V(\ov\Omega_1\cup\ov\Omega_2,\tau) - V(\ov\Omega_1,\tau)$.

\begin{figure}[ht]
\centerline{\hbox{\includegraphics[width=6cm]{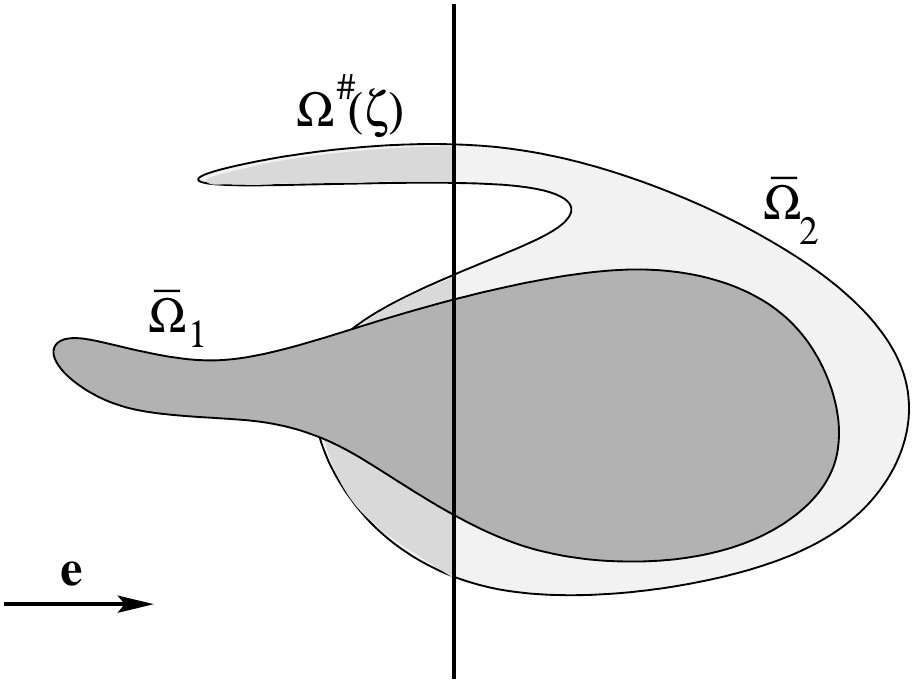}}}
\caption{\small The sets $\Omega^\sharp(\zeta)$ in (\ref{omsharp}), interpolating between $\ov\Omega_1$ and
$\ov\Omega_1\cup\ov\Omega_2$.}
\label{f:csm8}
\end{figure}

Fix a unit vector $\bfe\in\R^2$ and define the intermediate sets
\bel{omsharp}\Omega^\sharp(\zeta)~\doteq~\ov \Omega_1 \cup\Big\{ x\in \ov\Omega_2\,;~\langle \bfe, x\rangle \leq \zeta\Big\}.\eeq
(see Fig.~\ref{f:csm8}).  Observe that 
\bel{bome}\H^1\bigl(\partial \Omega^\sharp(\zeta)\bigr)~\leq~\H^1(\partial \ov\Omega_1\cup \partial \ov\Omega_2)~\leq~2\kappa.\eeq
Indeed, intersecting any set $S\subset\R^2$ with a half space, the length of its boundary
does not increase.

Define the increasing map $t\mapsto \zeta(t)$ implicitly by 
$$\caL^2\bigl(\Omega^\sharp(\zeta(t)\bigr)~=~\caL^2\bigl(\ov\Omega_1\bigr)+(\tau'-t)\,,
\qquad\qquad \tau\leq t\leq \tau'\doteq\tau + {\delta\over M-2\kappa}\,.$$
By (\ref{bome}) the map $t\mapsto \Omega^\sharp(\zeta(t))$ is admissible.
Moreover, 
$$\Omega^\sharp\bigl(\zeta(\tau)\bigr)~=~\ov\Omega_1\cup\ov\Omega_2\,,\qquad\qquad
\Omega^\sharp\bigl(\zeta(\tau')\bigr)~=~\ov\Omega_1\,.$$
This implies
$$\bega{l}
\ds V(\tau,\ov\Omega_2) ~\leq~V(\tau,\ov \Omega_1\cup\ov\Omega_2) 
~\leq~\kappa\int_\tau^{\tau'} \caL^2 \big(\Omega^\sharp(\zeta(t))\big)\, dt + V(\tau', \ov\Omega_1)-\caL^2(\ov\Omega) \\[4mm]
 \quad \ds ~\leq~(\tau'-\tau)\Big(\caL^2(\ov\Omega_1)+{\delta\over M-2\kappa}\Big) + V(\tau',\ov\Omega_1).
\enda$$
Since the difference $V(\tau', \ov\Omega_1)- V(\tau, \ov\Omega_1)$ was already estimated in step {\bf 1},
this achieves the proof.
\endproof

\section{A counterexample}
\label{sec:3}
\setcounter{equation}{0}

Next, we  construct a counterexample showing that,
if the assumption $M>2\kappa$ is dropped,
the value function can be discontinuous  w.r.t.~the 
distance (\ref{DFK}) on the initial data.
As usual, by $B(x,r)$ we denote the open ball centered at $x$ with radius $r>0$.
\begin{proposition}
Let $\kappa=2\pi $ and $M<\pi$. Then there exists 
a uniformly bounded sequence of initial data $\ov\Omega_n$   
such that
$$\H^1(\partial \ov\Omega_n)~=~2\pi,\qquad\qquad \caL^2(\ov \Omega_n)\,=\,
{\pi\over n^2}$$
for every $n\geq 1$. Moreover, for any admissible motion $t\mapsto \Omega_n(t)$ one has
$$\liminf_{n\to \infty} \caL^2(\Omega_n(T)\bigr)~\geq~{M^2\over 4\pi}
\qquad\hbox{for every}~~T>0.$$
\end{proposition}

{\bf Proof.} {\bf 1.} As a preliminary, observe that the area $A(t) = \caL^2\bigl(\Omega(t)\bigr)$ 
of an admissible moving set satisfies 
$${d\over dt} A(t)~\geq ~-M + \H^1\bigl(\partial \Omega(t)\bigr)~\geq~-M + 2\sqrt{ \pi A(t)},$$
where we used the isoperimetric inequality.  Therefore, if at some time $\tau$ we have
$A(\tau)~\geq~{M^2/ 4\pi}$, then 
$$A(t)\geq {M^2\over 4\pi}\qquad\forall t\geq \tau.$$
\v
{\bf 2.} We now consider an initial set $\ov\Omega_n$ consisting of the union of $n^2$ disjoint discs,
all with the same radius $r= n^{-2}$, centered at the points
$$P_{ij}~=~\Big( {iR\over 4n}, {jR\over 4n}\Big).$$
Notice that, if  the discs expand with unit speed in all directions, they will not touch each other until time $\tau={R\over 8n}-{2\over n^2}$.
\v
{\bf 3.} We now establish a lower bound on the area  $\caL^2\bigl(\Omega(\tau)\bigr)$.
Consider the following strategy:
On the initial time interval $[0,t_1]$
we concentrate all the effort on one disc until it is entirely wiped out.
On the second time interval $[t_1, t_2]$  we concentrate all the effort on a second disc, and so on.

We first compute the terminal area $\caL^2\bigl(\Omega(T)\bigr)$ determined by this strategy,
then we will show that this is actually the optimal one.

%
%
%
%

Call $r(t)$ the radius of a disc where all the effort is concentrated.  For $t\in [t_{k-1}, t_k]$,
this provides a solution to the Cauchy problem
$$\dot r~=~1-{M\over 2\pi r}\,,\qquad\qquad r(t_{k-1}) = {1\over n^2} + t_{k-1}\,.$$
Hence
$$r(t)-r(t_{k-1})-{M\over2\pi}\log\left({M-2\pi r(t_{k-1})\over M-2\pi r(t)}\right)~=~t-t_{k-1}\,.$$
In particular, this yields
\bel{tkdef}r(t_k)\,=\,0 \qquad\Longrightarrow \qquad t_k~=~t_{k-1}-r(t_{k-1})-{M\over2\pi}\log\left(1-{2\pi\over M}r(t_{k-1})\right).\eeq
\v
{\bf 4.}
Calling
$r_k~=~n^{-2} + t_k$, by (\ref{tkdef}) we have
\bel{stk}r_k~=~r_{k-1}+ (t_k-t_{k-1})~=\,- {M\over 2\pi} \ln\left( 1-{2\pi\over M} r_{k-1}\right).\eeq
We compare  the sequence \eqref{stk} with the solution of the Cauchy problem
$$
\begin{cases}
\dot{\zeta} = \dfrac{\pi}{M}\zeta^2, \\[6pt]
\zeta(0) = \dfrac{1}{n^2},
\end{cases}
$$
which is explicitly computed as
$$ 
\zeta(s) \,=\, \frac{1}{\,n^2- \tfrac{\pi}{M}s\,}.
$$
We observe that  the sequence $\zeta_k\doteq \zeta(k)$, $k=0,1,2,\ldots$, 
satisfies the inductive relations
$$\begin{cases}
\displaystyle \zeta_{k+1} = {\zeta_k \over 1-{\pi\over M}\zeta_k}, \\[8pt]
\zeta_0 = \dfrac{1}{n^2}\,.
\end{cases}$$
Moreover
$${x\over 1-{\pi\over M}x}~\leq \,  - {M\over 2\pi} \ln\left( 1-{2\pi\over M} x \right) \qquad \hbox{for }~
 0\leq x<{M\over 2\pi}\,.$$
A comparison with (\ref{stk}) thus yields
$$ \zeta(k)\leq r_k \qquad \hbox{ as long as}\quad r_k<{M\over 2\pi}\,.$$

At time $\tau= R/9n$, the  number $m$ of discs that have been eliminated satisfies 
$$\tau~\geq~{1\over n^2- {\pi\over M} m}\,,\qquad \qquad m ~\leq~\left(n^2-{1\over\tau}\right) \,{M\over \pi}
~=~\left( n- {9\over R}\right) {M\over \pi}\, n.$$
The number of discs that have NOT been eliminated is greater than 
$$\left[ n - \Big( n- {9\over R}\Big) {M\over \pi}\right] n~\geq~\left(1 - {M\over \pi}\right)n^2.$$

The total area of these discs is
$$A(\tau)~ \geq~\left(1 - {M\over \pi}\right)n^2 \cdot \pi  \left( {1\over n^2}+\tau\right)^2
~\geq~\left(1 - {M\over \pi}\right)n^2\cdot \pi  {R^2\over 81 n^2}=
\left(1 - {M\over \pi}\right)\cdot \pi  {R^2\over 81}\,.$$
As soon as 
$$\left(1 - {M\over \pi}\right)\cdot \pi  {R^2\over 81}
~>~{M^2\over 4\pi}, $$
by step {\bf 1} the total area cannot decrease.

\v
{\bf 5.}  It remains to check that the strategy of cleaning up one disc at a time is optimal.
This can be done by a comparison argument.  

Consider any admissible strategy:
$$\Omega(t)~=~\bigcup_{k=1}^{n^2} \Omega_k(t),$$
where the sets $\Omega_k(t)$ remain disjoint up to time $\tau={R\over 8n}-{2\over n^2}$.
Call $A_k(t)$, $P_k(t)$ respectively the area and the perimeter of $\Omega_k(t)$, and set
$$A(t)~=~\sum_k A_k(t)~=~ \caL^2\bigl(\Omega(t)\bigr),\qquad\qquad 
P(t)~=~\sum_k P_k(t)~=~ \H^1\bigl(\partial\Omega(t)\bigr).$$

Define the function $\Phi(t,A)$ to be the infimum of the following constrained isoperimetric problem:
$$\hbox{minimize:} ~~\sum_{k=1}^{n^2} \H^1(\partial U_k),$$
where the sets $U_1,\ldots, U_{n^2}$ satisfy the constraints
$$\hbox{subject to:}   ~~\sum_i \caL^2(U_k) = A,\qquad \caL^2(U_k)~\leq ~\pi\left({1\over n^2}+t\right)^2\quad\forall i=1,\ldots, n^2.$$
We observe that this minimum is obtained when:
\begi
\item every set $U_k$ is a disc (possibly empty),
\item all discs $U_k$, except at most of one of them, have radius either $r= n^{-2} +t$ or $r=0$.
\endi
Indeed, if two sets, say $U_1,U_2$ are discs with radii $0<r_1\leq r_2< n^{-2}+t$,
we can enlarge $U_2$ and shrink $U_1$ so that the sum of the areas remains constant.
In  this case, the sum of the perimeters will decrease.

We now have
\bel{DA1}{d\over dt} A(t)~=~P(t) -M~\geq~\Phi\bigl(t, A(t)\bigr) -M.\eeq
On the other hand, calling $A^*(t)$ and $P^*(t)$ respectively the area and the perimeter
of the set $\Omega(t)$ obtained by cleaning up one disc at a time, we have
\bel{DA2}{d\over dt} A^*(t)~=~P^*(t) -M~=~\Phi\bigl(t, A^*(t)\bigr) -M.\eeq
By comparing (\ref{DA1}) with (\ref{DA2}), since $\Phi$ is Lipschitz continuous, we conclude that 
$A(t)\geq A^*(t)$ for every $t\geq 0$.
\endproof

\section{Continuous dependence of the minimum slicing cost}
\label{sec:4}
\setcounter{equation}{0}
Given $r>0$ and a set $G$, we denote by $B(G,r)\doteq~\bigl\{ x\,;~d(x,G)<r\bigr\}$
the open neighborhood with radius $r$ around the set $G$.
We recall that the Hausdorff distance between two sets $G_1,G_2\subset\R^n$ is
$$\bega{rl} d_H(G_1, G_2)
&\doteq~\inf~\Big\{ r>0\,;~~G_1\subset B(G_2,r)~~\hbox{and}~~G_2\subset B(G_1,r)
\Big\}\\[4mm]
&=~\max\Big\{ \sup \bigl\{ d(x_1, G_2)\,;~x_1\in G_1\bigr\}~,~\sup\bigl\{d(x_2, G_1)\,;~x_2\in G_2\bigr\}\Big\}
.\enda$$

\begin{lemma}\label{l:41} 
If $G_1\subset G_2$ are bounded open sets with finite perimeter, then their slicing costs
satisfy $V(G_1)\leq V(G_2)$.
\end{lemma}

{\bf Proof.}  
Let $t\mapsto \Omega_2(t)$ be an optimal slicing of $G_2$. Set 
$$T_1= \caL^2(G_1),\qquad T_2= \caL^2(G_2).$$
Observe that the function $$s~\mapsto~\tau(s)~=~ \caL^2(\Omega_2(s)\cap G_1)$$
is nondecreasing.  Denote by
$$t~\mapsto~\pi(t)~\doteq ~\inf\big\{s : t<\tau(s)\big\} $$ the right inverse of $s\mapsto\tau(s)$, so that
$\tau\bigl(\pi(t)\bigr)=t$.
The strategy  
$$t~\mapsto~ \Omega_1(t)~=~\Omega_2(\pi(t))\cap G_1$$
describes a slicing of $G_1$: 
$$ \Omega_1(t_1)\subset  \Omega_1(t_2) \quad \text{ if }\, t_1<t_2,
\qquad \text{ and } \qquad \caL^2(\Omega_1(t))=t\,.$$
Moreover, its cost is lower  then the cost of $\Omega_2$. 
Indeed, applying the (possibly discontinuous) change of variable $s=\pi(t)$, we get
\begin{align*}
\int_0^{T_1}\H^1\bigl(\partial \Omega_1(t)\cap G_1\bigr)dt&=
\int_0^{T_1}\H^1\bigl(\partial \Omega_2(\pi(t))\cap G_1\bigr)dt
\leq\int_0^{\pi(T_1)}\H^1\bigl(\partial \Omega_2(s)\cap G_1\bigr)ds\\
&\leq\int_0^{T_2}\H^1\bigl(\partial \Omega_2(s)\cap G_1\bigr)ds\,.
\end{align*}
Since $G_1\subset G_2$,
$$\H^1\bigl(\partial \Omega_2(s)\cap G_1\bigr)~\leq~\H^1\bigl(\partial \Omega_2(s)\cap
G_2\bigr)\,,$$
we finally obtain that
$$V(G_1)\leq \int_0^{T_1}\H^1\bigl(\partial \Omega_1(t)\cap G_1\bigr)dt\leq
\int_0^{T_2}\H^1\bigl(\partial \Omega_2(t)\cap G_2\bigr)dt=V(G_2)\,.$$
\endproof

Our main result on the continuous dependence of the slicing cost is based on the Hausdorff distance among boundaries.  
Some regularity of the domains $G\subset \R^2$ needs to be imposed.
We recall here the main definition \cite{HMT}.  Here $G^c$ denotes the complement $\R^n\setminus G$ of the set $G$, while $\ov G$ denotes the closure.
\begin{definition}
Let $G$ be a nonempty, bounded open subset of $\R^n$.  We say that $G$ is a {\bf strongly Lipschitz domain}  if, for every $x_0\in \partial G$, there exist constants $b, c ,d> 0$ for which the following holds. There exist a hyperplane  $H$  through $x_0$, a  unit normal $\bfn$ to $H$, and an open cylinder 
$$\Gamma_{b,c} ~\doteq~ \Big\{x' + t \bfn\,;~~ x' \in H, ~|x' - x_0| < b,~~ |t| < c\Big\}$$ 
such that
\bel{cilinders}  \bega{rl} \Gamma_{b,c}\cap G&=~ \Gamma_{b,c}\cap\bigl\{x' + t\bfn\,;~~x'\in H,~t<\vp(x')\bigr\},\\[2mm]
\Gamma_{b,c}\cap \partial G&=~\Gamma_{b,c}\cap\bigl\{x' + t\bfn\,;~~x'\in H,~t=\vp(x')\bigr\},\\[2mm]
\Gamma_{b,c}\cap \ov G^{\,c}&=~ \Gamma_{b,c}\cap\bigl\{x' + t\bfn\,;~~x'\in H,~t>\vp(x')\bigr\},
\enda \eeq
for some Lipschitz function $\vp:H\mapsto \R$ satisfying
\bel{vpprop} \vp(x_0)\,=\,0\qquad \hbox{and}\qquad \bigl|\vp(x')\bigr|< d~~\hbox{if}~~|x'-x_0|\leq b.\eeq
\end{definition}

A closely related property is the following.
Denote by $\langle\cdot,\cdot\rangle$ the Euclidean inner product on $\R^n$.
We say that a bounded open set $G\subset\R^n$ satisfies the {\bf uniform cone property} if there exists an open truncated 
cone  of the form
$$C\,\doteq\,\Big\{ x\in\R^2\,;~~~c_0 \langle \bfe, x\rangle<|x|<\rho\Big\}$$ 
for some $c_0, \rho>0$ and some unit vector $\bfe$,
such that the following holds. 
For every $x_0\in \partial\Omega$
there exists $r>0$ and a rotation ${\cal R}$  such that 
$$x+{\cal R}(C)\,\subset \,G\qquad\forall \quad x\in B(x_0, r)\cap \ov G.$$
As proved in Theorem 1.2.2.2  of \cite{G}, one has
\begin{lemma} 
A bounded open set  is strongly Lipschitz if and only if it satisfies the uniform cone property.
\end{lemma}

\begin{figure}[ht]
\centerline{\hbox{\includegraphics[width=12cm]{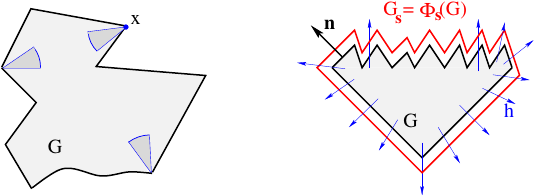}}}
\caption{\small Left: a strongly Lipschitz domain $G$, satisfying the
uniform cone property. Right: a vector field $h$, transversal to the boundary of $G$ at every point $x\in \partial G$.   For every $s>0$ small enough,  
the image of the set $G$ under the map $x\mapsto \Phi_s(x)=x+ s h(x)$ covers a whole neighborhood of $G$.}
\label{f:csm20}
\end{figure}

An important property of Lipschitz domains is that they admit smooth  vector fields which are transversal to their boundary.
By Proposition~2.3 and Theorem~2.7 of \cite{HMT}, one has
\begin{lemma} Let the bounded open set $G\subset\R^n$ be a strongly Lipschitz domain.
Then there exists $\kappa>0$ and a vector field $h\in \C^\infty(\R^n;\R^n)$ such that
\bel{transv} \la \bfn(x),\, h(x)\ra\,\geq\,\kappa\qquad\hbox{for a.e.} ~x\in\partial G.\eeq
Here $\bfn (x)$ denotes the unit outer normal to the boundary $\partial G$ at the point $x$, while
``almost everywhere" refers to the $(n-1)$-dimensional Hausdorff measure on
$\partial G$.
\end{lemma}

For the application we have in mind,  a  transversal vector field can be used 
to construct a smooth homotopy
of $G$ onto a strictly larger or a strictly smaller domain.    For $s\in \R$, define
the map $\Phi_s:\R^n\mapsto\R^n$ by setting
\bel{Fs}
\Phi_s(x)~\doteq~x+s h(x).\eeq
Moreover, consider the images
\bel{Gs}
G_s~\doteq~\bigl\{ x+ s h(x)\,;~~x\in G\bigr\}.\eeq

\begin{lemma}\label{l:44} Let $G\subset\R^n$ be a strongly Lipschitz domain and let $h\in \C^1(\R^n;\R^n)$ 
be a  transversal vector field, as in (\ref{transv}).
Then, for every $\ve>0$  there exists $s_0>0$ such that the following holds.

For every $s\in \,[- s_0, s_0]$ the $\C^1$ map $x\mapsto \Phi_s(x)$ has  a $\C^1$ inverse $\Phi_s^{-1}$.
Denoting by $I$ the identity map, one has
\bel{F1} \bigl\|\Phi_s-I\bigr\|_{\C^1}\,<\,\ve,\qquad\qquad \bigl\|\Phi^{-1}_s-I\bigr\|_{\C^1}\,<\,\ve.\eeq
\bel{det}\Big| \det\bigl(D\Phi_s(x)\bigr) - 1\Big|\,<\,\ve,\qquad \Big| \det\bigl(D(\Phi_s)^{-1}(x)\bigr) - 1\Big|~<~\ve\qquad\quad \forall ~x\in \R^2.\eeq
Moreover we have the implications
\bel{GGs} s<0\quad\implies\quad \ov G_s\subset G,\qquad\qquad \qquad 
s>0\quad\implies\quad \ov G\,\subset\,G_s\,.\eeq
\end{lemma}
{\bf Proof.}
The bounds (\ref{F1}) are an easy consequence of the implicit function theorem, and 
remain valid
for any vector field whose $\C^1$ norm is bounded.
For  a proof of the last statement we refer to
 Proposition~4.19 in \cite{HMT}.\endproof

We are now ready to prove the main result of this section, on the continuous dependence of the 
slicing cost within the family of strongly Lipschitz domains.  We shall use the notation
\bel{G-r}B(G, -r)~\doteq~\bigl\{ x\,;~B(x,r)\subseteq G\bigr\}.\eeq

\begin{theorem}\label{t:41} Let  $G$ be a strongly Lipschitz domain.
For any $\ve>0$ there exists $r>0$ such that, if $G'$ is another bounded open set such that
\bel{near} B(G, -r)~\subseteq ~G'~\subset~B(G,r),\eeq
then the minimum slicing costs satisfy
\bel{JGG} \bigl| V(G') - V(G)\bigr|~\leq~\ve.\eeq
\end{theorem}

{\bf Proof.}  
{\bf 1.} Since $G$ is strongly Lipschitz, there exists  a transversal  vector field
 $h\in \C^1(\R^2;\,\R^2)$.  
Denote by $\Phi_s$ the corresponding maps in (\ref{Fs}), and by $G_s$ the sets in 
(\ref{Gs}). 

By Lemma~\ref{l:44} we can choose $s_0>0$ small enough so that for  $s\in \,] -s_0, s_0]$ 
the bounds (\ref{F1})-(\ref{GGs}) hold.
\v
{\bf 2.}
Let  $t\mapsto \Omega(t)$, $t\in [0,T]$  be an 
optimal slicing of the set $G$. Consider the images
$$\Omega_s(t)~\doteq~\Phi_s\bigl(\Omega(t)\bigr),\qquad\qquad t\in [0,T].$$
Since $\Phi_s$ is a $\C^1$ diffeomorphism, 
by (\ref{F1}) the lengths of the boundaries satisfy
\bel{Fbo} \H^1\bigl(G_s\cap \partial \Omega_s(t)\bigr) ~\leq~\bigl\|\Phi_s\bigr\|_{\C^1} \cdot 
 \H^1\bigl(G\cap \partial \Omega(t)\bigr)~\leq~(1+\ve)\, \H^1\bigl(G\cap \partial \Omega(t)\bigr).
 \eeq
%
%
In addition, for any $0<t_1<t_2<T$,  by the first inequality in (\ref{det})  it follows
\bel{55}\Big|\caL^2\bigl(\Omega_s(t_2) \setminus \Omega_s(t_1)\bigr)- \caL^2\bigl(\Omega(t_2) \setminus \Omega(t_1)\bigr)\Big|~\leq~\ve \caL^2\bigl(\Omega(t_2) \setminus \Omega(t_1)\bigr)~=~
\ve(t_2-t_1).\eeq
\v
{\bf 3.}
To achieve a slicing of $G_s$, we need to perform a time rescaling.
Calling $T_s = \caL^2(G_s)$, define the map $\phi_s:[0,T]\mapsto [0, T_s]$ by setting
$$\phi_s(t)  ~\doteq~\caL^2\bigl(\Omega_s(t)\bigr).$$
By (\ref{55}) this map is Lipschitz continuous, strictly increasing, together with its inverse $\phi_s^{-1}$.   By the second inequality in (\ref{det}),
\bel{phin}\left| {d\over d\tau} \phi_s^{-1}(\tau)-1\right|\,<\,\ve\qquad\qquad\hbox{for a.e.}~\tau\in [0,T].\eeq
\v
{\bf 4.} By the previous construction, the map
$$\tau~\mapsto~\Hat \Omega(\tau)~\doteq~\Omega_s\bigl(\phi_s^{-1}(\tau)\bigr)$$
provides a slicing of the set $G_s$.
Choosing $s_0$ sufficiently small, for  any $s\in [0, s_0]$ 
by  (\ref{Fbo}) and  (\ref{phin}) we thus obtain
\bel{56}
\bega{rl} J(\Hat\Omega)&\ds=~\int_0^{T_s} \H^1\bigl( G_s\cap \partial\Hat\Omega(\tau)\bigr)\, d\tau
~=~\int_0^{T} \H^1\bigl( G_s\cap \partial\Hat\Omega(\phi^{-1}_s(t))\bigr)\, {d\tau\over dt} \, dt
\\[4mm]
&\ds\leq~\int_0^{T} (1+\ve)\H^1\bigl( G\cap \partial\Omega(t)\bigr)\cdot (1+\ve) \, dt
~\leq~(1+\ve)^2 V(G).
\enda
\eeq
This implies $V(G_s)\leq (1+\ve)^2 V(G)$.

An entirely similar analysis applies to the map $(\Phi_{-s})^{-1}: G_{-s}\mapsto G$.
For $s\in [-s_0, 0]$, this yields $V(G)\leq (1+\ve)^2 V(G_{-s})$.
\v
{\bf 5.} Now choose any $s\in \,]0, s_0]$.  Since the sets $G, G_s, G_{-s}$ are open, by (\ref{GGs}), a compactness argument implies that there exists a radius $r>0$ such that
$$B(G_{-s}, r)\subset G,\qquad\quad B(G,r)\subset G_s\,.$$
Therefore,  (\ref{near}) implies
$G_{-s}\subset G'\subset G_s$.  Hence, by Lemma~\ref{l:41},
\bel{VG'}
{1\over (1+\ve)^2} V(G)~\leq~V(G_{-s})~\leq~V(G')~\leq~V(G_s)~\leq~(1+\ve)^2 V(G).\eeq
Since $\ve>0$ was arbitrary, this proves the theorem.
\endproof

\begin{remark} {\rm Consider again the minimum time problem {\bf (MTP)},
where the control effort satisfies the constraint 
$$ \E(t)~\doteq~
\int_{\partial\Omega(t)\cap G}E\bigl(\beta(t,x)\bigr)\, \H^1(dx)~\leq~M\qquad\qquad\forall t>0.$$
By an entirely similar analysis one can establish the continuity of the eradication
time, depending on the domain $G$.}
\end{remark}

\begin{figure}[ht]
\centerline{\hbox{\includegraphics[width=8cm]{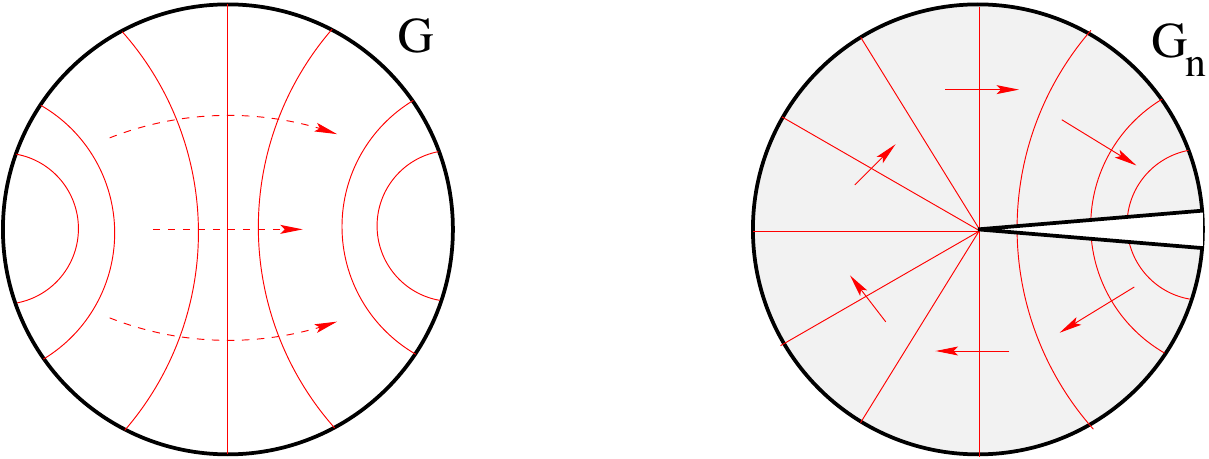}}}
\caption{\small  For the two sets $G, G_n$, the distance 
$\bigl\| {\bf 1}_{G_n}- {\bf 1}_G\bigr\|_{\L^1}$  between the
 characteristic functions is small, but the 
Hausdorff distance between the boundaries  is large.   The slicing costs are much different.}
\label{f:csm17}
\end{figure}

\begin{example} {\rm 
As shown in Fig.~\ref{f:csm17},  let $G$ be the open unit disc, while $G_n$ is the unit disc minus a sector  with angular opening $\theta_n= 1/n$.  As $n\to\infty$ we have
$$\caL^2(G_n)\,\to\,\caL^2(G),\qquad\qquad d_H(G_n,G)\,\to\,0.$$
However, the slicing costs do not converge:
$$\limsup_{n\to\infty} V(G_n)~<~V(G).$$
Notice that in this example the Hausdorff distance between the boundaries  does not decrease to zero. Indeed,
$d_H(\partial G_n,\partial G)=1$ for all $n$.   }
\end{example}

\section{A conjecture on the optimal slicing of a disc}
\label{sec:5}
\setcounter{equation}{0}

\begin{conjecture} {\bf (slicing conjecture).}  \label{c:51}
Among all bounded open sets $G\subset\R^2$ with unit area, the disc
has the maximum slicing cost.
\end{conjecture}

More precisely, calling $V(G)$ the minimum slicing cost of the set $G$ and $V(D)$ the minimum slicing cost of a disc $D$ with unit area,
we conjecture that
\bel{1}V(G)~\leq~V(D)~=~
{2\over 3\sqrt\pi}(1+\log4),\eeq
where equality is achieved if and only if $G$ is a disc. 
An explicit computation of the slicing cost for a disc can be found in the Appendix.
For sets with arbitrary area, a rescaling argument yields
\bel{111}V(G)~\leq~\bigl[\caL^2(G)\bigr]^{3/2} \cdot 
{2\over 3\sqrt\pi}(1+\log 4).\eeq

Conjecture~\ref{c:51} is closely related to the well known isoperimetric conjecture \cite{students, WW},
which we now recall.
Given an open set $G\subset\R^2$ with $\meas(G)=1$, for any $\lambda\in [0,1]$ define
\bel{psil}
I_G(\lambda)~\doteq~\min\Big\{ \H^1(\partial S\cap G)\,;~~S\subseteq G,~~\meas(S)=\lambda\Big\}.\eeq
The function $\lambda\mapsto I_G(\lambda)$ is called the {\bf isoperimetric profile} of the set $G$.

\begin{conjecture}{\bf (isoperimetric conjecture).} \label{c:52}
Let $D$ be a disc with unit area. Then, for any convex set $G\subset\R^2$ with unit area, one has
\bel{isoco}
 I_G(\lambda)~\leq~I_D(\lambda)\qquad\qquad\forall ~\lambda\in [0,1].\eeq
\end{conjecture}

This conjecture has been proved in [14] in two main cases: (i) $G$ is a small perturbation of
the disc, and (ii) $G$ is a smooth convex domain, symmetric w.r.t.\ both coordinate axes, and the
curvature of its boundary $\partial G$ has exactly two local maxima and two local minima. In the case
of regular polygons, the conjecture was proved in [3]. For a general convex set $G$ it remains
open. We recall that, for the related (single-cut) problem of bisecting a convex domain by
the shortest possible curve was fully proved in
\cite{EFKNT}.
\v

To see how (\ref{isoco}) can help in proving the slicing conjecture,  assume that $G$ 
has unit area and admits a 
slicing $t\mapsto \Omega(t)$ such that 
\bel{osli}\H^1\bigl(\partial\Omega(t)\bigr) ~=~I_G(t)\qquad\qquad \forall t\in [0,1].\eeq
Then by the sufficient conditions it follows that $\Omega(\cdot)$ is optimal.
On the other hand, if the isoperimetric conjecture holds true, then (\ref{isoco}) yields
\bel{VGB} V(G)~=~\int_0^1 \H^1\bigl( \partial \Omega(t)\bigr)\, dt~=~\int_0^1 I_G(t)\, dt ~\leq~\int_0^1 I_D(t)\, dt
~=~V(D).\eeq

One can regard the slicing conjecture as a dynamic counterpart to the isoperimetric conjecture.
However, a positive solution to the  isoperimetric conjecture would not yield a straightforward 
proof of the
slicing conjecture.    Indeed, since every slicing strategy must satisfy the monotonicity property 
(\ref{Om1}), for some sets $G$ the optimal slicing may not satisfy the identity (\ref{osli}).
In this case, the argument (\ref{VGB}) breaks down.

In the remainder of this section we prove that the slicing conjecture holds for a various classes
of sets $G$.

\begin{figure}[ht]
\centerline{\hbox{\includegraphics[width=5cm]{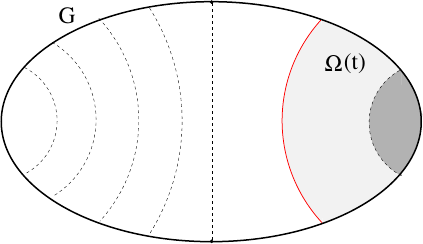}}}
\caption{\small  An optimal slicing of an ellipse $G$.  Here the relative boundaries $\partial \Omega(t)\cap G$ are arcs of circumferences that cross the boundary $\partial G$ perpendicularly.}
\label{f:sli12}
\end{figure}

\subsection{The slicing cost of an ellipse.}   
In this subsection we prove that the slicing conjecture is true for all ellipses.
\begin{proposition}\label{p:51} Let $G$ be an ellipse and let $D$ be a disc with the same area.
Then their slicing costs satisfy $V(G)\leq V(D)$.
\end{proposition} 

{\bf Proof.} {\bf 1.} 
To fix ideas, assume $a<b$ and consider the ellipse
\bel{Geq} G~=~\left\{ (x,y)\,;~~{x^2\over a^2} + {y^2\over b^2}~<~1\right\}.\eeq
The set $G$
 is symmetric w.r.t.~both coordinate axes. Moreover, the curvature of its boundary $\partial G$
has exactly two minima and two maxima.  By the results in \cite{WW}
the set $G$ thus satisfies the inequalities (\ref{isoco}).
To conclude that $G$ satisfies the slicing conjecture, it thus suffices to show that $G$ admits a slicing 
$t\mapsto\Omega(t)$ such that (\ref{osli}) holds.  

As shown in  Fig.~\ref{f:sli12}, we define $\Omega(t)\subseteq G$ to be a set with area $t$,
whose relative boundary $G\cap \partial \Omega(t)$ is an arc of circumference, crossing
$\partial G$ perpendicularly at both endpoints.

To prove that this is indeed an optimal slicing, it remains to show that these arcs of circumferences
do not intersect each other.
\v
{\bf 2.} For $x\in [-a,a]$, consider the function parameterizing the upper portion of the ellipse:
\bel{phidef}y~=~\phi(x)~\doteq~b\sqrt{1-{x^2\over a^2}}\,,
\qquad\qquad \phi'(x)~=~{-bx\over a \sqrt{a^2-x^2}}\,.\eeq

\begin{figure}[ht]
\centerline{\hbox{\includegraphics[width=9cm]{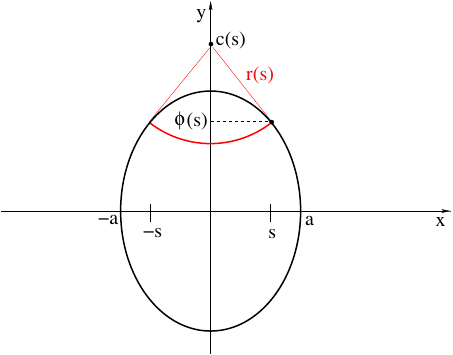}}}
\caption{The functions $\phi(s)$, $c(s)$ and $r(s)$,  
used in the proof of Proposition~\ref{p:51}.}\label{f:sli13}
\end{figure}

For every $s\in\,]0,a[\,$ there is a circumference with center along the $y$-axis which crosses 
$\partial G$ perpendicularly at the points $\bigl(s, \phi(s)\bigr)$ and $(-s, \phi(s))$. 
Its center is at the point $(0, c(s))$, where
\bel{cs} c(s)\,=\,\phi(s)-s\,\phi'(s)~=~b\sqrt{1-{s^2\over a^2}}+  {bs^2\over a \sqrt{a^2-s^2}}
~=~{ab\over \sqrt{a^2-s^2}}\,.\eeq
Its radius is
\bel{rs}r(s)\,=\,s\sqrt{1+ [\phi'(s)]^2}~=~s\sqrt{ 1+ {b^2 s^2\over a^2(a^2-s^2)}}\,.\eeq
For any fixed $s$, we parameterize the corresponding arc of circumference by setting
$$\varphi_s(x)\,=\,c(s)-\sqrt{[r(s)]^2-x^2},\quad\qquad x\in\,]-s,s[\,.$$
To prove that these arcs of circumferences do not intersect each other, it suffices to show that
$${d\over ds}\varphi_s(x)\,<\,0\qquad\quad\forall  x\in\,]-s,s[\,.
$$
\v
{\bf 3.}
Observing that, for any  $x\in\,]-s,s[\,$,
$${d\over ds}\varphi_s(x)~=~c'(s)-{r(s)r'(s)\over\sqrt{[r(s)]^2-x^2}}~\leq~
c'(s)-r'(s)~=~{d\over ds}\varphi_s(0)\,,$$
it is enough to prove that
\bel{dpneg}{d\over ds}\varphi_s(0)~<~0\qquad\qquad \forall s\in \,]0,a[\,.\eeq

Toward this goal, we perform the change of variable $\sigma=s/ a\in [0,1]$. 
Recalling (\ref{cs})-(\ref{rs}), we compute
$$\bega{l}\ds {1\over b} \cdot {d\over d\sigma} \bigl( c(a\sigma)) - r(a\sigma)\bigr)
~=~\ds{d\over d\sigma}\left({1\over\sqrt{1-\sigma^2}}-\sigma
\sqrt{ {a^2\over b^2}+{\sigma^2\over 1-\sigma^2}}\right)
\\[4mm]
\qquad =~\ds{d\over d\sigma}\left({1\over\sqrt{1-\sigma^2}} -
{\sigma\over\sqrt{1-\sigma^2}}\sqrt{{a^2\over b^2}\big(1-\sigma^2\big)+\sigma^2}\right)
\\[4mm]
\qquad\ds =~{\sigma\over\big(1-\sigma^2\big)^{3/2}}-
{1\over\big(1-\sigma^2\big)^{3/2}}\sqrt{{a^2\over b^2}\big(1-\sigma^2\big)+\sigma^2}
  - {1\over\sqrt{1-\sigma^2}}\sqrt{{a^2\over b^2}\big(1-\sigma^2\big)+\sigma^2}\\[4mm]
 \ds \qquad\qquad -
{\sigma\over\sqrt{1-\sigma^2}}{\sigma\left(1-{a^2\over b^2}\right)\over\sqrt{{a^2\over b^2}\big(1-\sigma^2\big)+\sigma^2}}
\\[4mm]
\ds\qquad
<~{1\over\big(1-\sigma^2\big)^{3/2}}\cdot \left( \sigma-
\sqrt{{a^2\over b^2}\big(1-\sigma^2\big)+\sigma^2}\right)~\leq ~0\,.
\enda
$$
Going back to the original variables, this establishes the inequality (\ref{dpneg}),
completing the proof.   \endproof

\begin{figure}[ht]
\centerline{\hbox{\includegraphics[width=10cm]{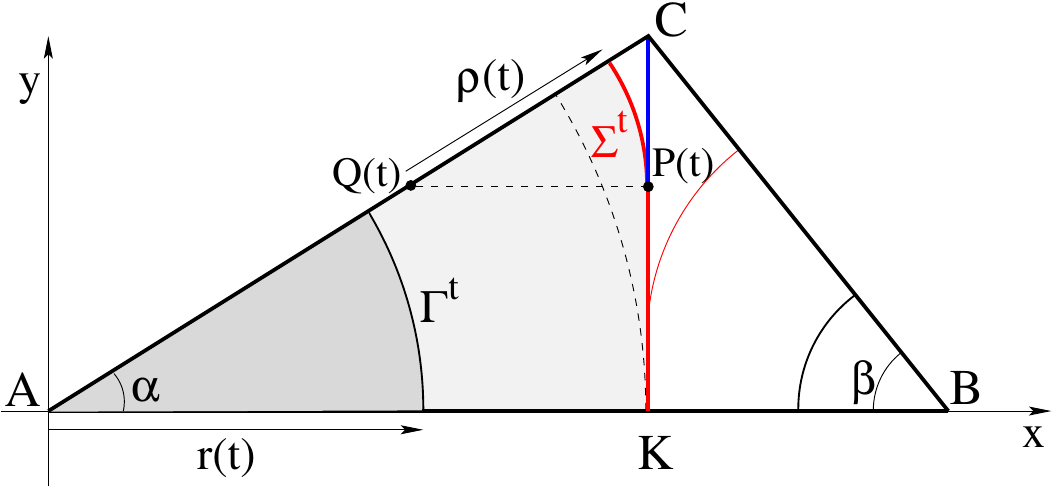}}}
\caption{
Constructing a suboptimal slicing of a triangle.}\label{f:sli11}
\end{figure}

\subsection{The slicing cost of a triangle.}   
We now prove that the slicing conjecture is also true for all triangles.
\begin{proposition} Let $G$ be a triangle and let $D$ be a disc with the same area.
Then their slicing costs satisfy $V(G)\leq V(D)$.
\end{proposition} 

{\bf Proof.} {\bf 1.} Let  $\caL^2(G)=\caL^2(D)=T$. We will construct a (non optimal)
slicing $t\mapsto \Omega(t)$ of the triangle $G$ such that
$$J\,\doteq\,\int_0^T\H^1\bigl( G\cap\partial \Omega(t)\bigr)\, dt~<~V(D).$$ 

Referring to Fig.\ref{f:sli11}, let $A,B,C$ be the vertices of the triangle $G$ and let $\alpha,\beta,\gamma$ be the corresponding angles. We assume that these angles satisfy
$0<\alpha\leq\beta\leq\gamma$. 

A slicing $t\mapsto \Omega(t)\subseteq G$ is constructed as follows.  
Call $K$ the perpendicular projection of $C$ on the segment $AB$.
Let $T = \caL^2(G)$ be the area of the triangle $ABC$ and call $T_1$ the area of the smaller triangle
$AKC$.
Moreover, call $\bar t$ the area of the intersection of the triangle $G$ with the disc centered at $A$
with radius $|K-A|$.    Call $\hat t$ the  area of the intersection of the triangle $G$ with the disc centered at $B$
with radius $|K-B|$.   

The slicing $t\mapsto \Omega(t)$ can now be defined as follows:
\begi
\item
For $t\in [0,\bar t]$, the set $\Omega(t)$ is a region  with area $t$ inside the triangle, which is bounded 
by an arc of circumference $\Gamma^t$ centered at $A$ (the dark shaded region in Fig.~\ref{f:sli11}).

\item For $t\in [\bar t, T_1]$, $\Omega(t)$  is a region  with area $t$ inside the triangle (the lightly shaded region in Fig.~\ref{f:sli11}),
bounded by  a curve $\Sigma^t$.  Here $\Sigma^t$ is the union of a vertical segment 
$K P(t)$ and an arc of circumference which is tangent to $KC$ at some point $P(t)$ and crosses the 
side $AC$ perpendicularly.  This implies that the center of the circumference $Q(t)= 
\bigl(x(t), y(t)\bigr)$ is a point on the segment $AC$.  We call $\rho(t)$ the radius of this circumference.
\item For $t\in [T_1, T-\hat t]$,  $\Omega(t)$ is a region bounded by a curve $\eta(t)$, where $\eta(t)$ is the union of a vertical segment 
$K P(t)$ and an arc of circumference which is tangent to $KC$ at $P(t)$ and crosses the 
side $BC$ perpendicularly.
\item For $t\in [T-\hat t, T]$,  $\Omega(t)$ is a region  with area $t$ inside the triangle, 
bounded by an arc of circumference 
centered at $B$.
\endi
\v
{\bf 2.} We now compute the cost of the above slicing:
$$J(\Omega)~=~J_1+J_2~=~\int_0^{T_1}\H^1\bigl( G\cap\partial \Omega(t)\bigr)\, dt+
\int_{T_1}^T\H^1\bigl( G\cap\partial \Omega(t)\bigr)\, dt\,.$$
The first integral can be split as
$$J_1~=~\int_0^{\bar t}\H^1\bigl(\xi(t)\bigr)\, dt+\int_{\bar t}^{T_1}\H^1\bigl(\eta(t)\bigr)\, dt\,.$$
For $t\in [0, \bar t]$, calling $r(t)$ the radius of the arc of circumference $\Gamma^t$, we find
$$t~=~\caL^2(\Omega(t))~=~{\alpha r^2(t)\over 2},\qquad 0\leq r(t)\leq |K-A|.$$
The length of this arc is
$$\H^1(\Gamma^t) ~=~\alpha r(t).$$
Introducing the area function
$$\A(r)\,\doteq~{\alpha r^2\over 2}, \qquad\qquad {d\over dr} \A(r)\,=\, \alpha r,$$
and recalling that $\A\bigl(r(t)\bigr)=t$, 
by changing the variable of integration we obtain
\bel{J11}\int_0^{\bar t}\H^1\bigl(\Gamma^t\bigr)\, dt~
=~\int_0^{\bar t}\alpha r(t) \, dt~=~\int_0^{x_C}\alpha r \cdot {d\A\over dr} \, dr ~=~
\int_0^{x_C}\alpha^2 r^2 \,dr~=~{\alpha^2\over3}x_C^3\,.\eeq
Here and in the sequel $x_C$ denotes  the $x$-coordinate of the vertex $C$.
\v
Next for $t\in [\bar t, T_1]$,  calling
$x(t)$ the $x$-coordinate of the point $Q(t)$, we find
$$t~=~\caL^2(\Omega(t))~=~{\alpha \bigl(x_C-x(t)\bigr)^2\over 2}+{\bigl(2x_C-x(t)\bigr)x(t)\tan\alpha\over2},\qquad \quad x(t)\in[0,x_C].$$
In this case, the length of the relative boundary $G\cap \partial \Omega(t)$ is computed by
$$\H^1(\Sigma^t)\,=\, x(t)\tan \alpha + \alpha \bigl(x_C- x(t)\bigr).$$
Consider the new area function
$$\A(x)\,\doteq~{\alpha \bigl(x_C-x\bigr)^2\over 2}+{\bigl(2x_C-x\bigr)x\tan\alpha\over2}, \qquad\qquad {d\over dx} \A(x)\,=\, (x_C-x)(\tan\alpha-\alpha),$$
defined for $x\in [0, x_C]$. 
Recalling that $\A\bigl(x(t)\bigr)=t$, 
by changing the variable of integration we now obtain
\bel{J12}\bega{l}\ds
\int_{\bar t}^{T_1}\H^1\bigl(\Sigma^t\bigr)\, dt~
=~\int_{\bar t}^{T_1}\Big( x(t)\tan \alpha + \alpha \bigl(x_C- x(t)\bigr)\Big)\, dt \\[4mm]
\ds\qquad=~\int^{x_C}_0 \Big( x\tan \alpha + \alpha \bigl(x_C- x\bigr)\Big)\, {d\A\over dx}\, dx\\[4mm]
\ds\qquad= ~\int^{x_C}_0\Big( x\tan \alpha + \alpha \bigl(x_C- x\bigr)\Big)(x_C-x)(\tan\alpha-\alpha)\,dx\\[4mm]
\ds\qquad =~x_C^3\left({1\over 6}\tan\alpha(\tan\alpha-\alpha)-{1\over3}\alpha^2\right).
\enda\eeq
Combining (\ref{J11}) with (\ref{J12}) we conclude
$$J_1~=~{\tan\alpha(\tan\alpha-\alpha)\over6}\,x_C^3\,.$$
By entirely similar calculations one obtains
$$J_2~=~{\tan\beta(\tan\beta-\beta)\over6}\,(x_B-x_C)^3\,.$$
\v
{\bf 3.} 
Finally, using the identities
$$T_1\,=\,{1\over 2}x_C^2\tan\alpha,\qquad  T-T_1\,=\,{1\over2}(x_B-x_C)^2\tan\beta, 
\qquad  x_C\tan\alpha=(x_B-x_C)\tan\beta,$$
we conclude that the total cost of the slicing strategy $t\mapsto\Omega(t)$ satisfies
$$
J(\Omega)~=~T^{3/2}\,{\sqrt2\over3}\left({\tan\alpha\tan\beta\over\tan\alpha+\tan\beta}\right)^{3/2}
\left({\tan\alpha-\alpha\over\tan\alpha^2}+{\tan\beta-\beta\over\tan\beta^2}\right)\,.
$$
We now observe that the inequalities $\alpha\leq\beta\leq \gamma$, $\alpha+\beta+\gamma=\pi$
imply
 $\alpha\leq \pi/3$.
In particular, $\tan\alpha\leq\sqrt{3}.$
Moreover, since $\alpha,\beta <\pi/2$, we have
$$
\frac{\tan\alpha-\alpha}{\tan^2\alpha}
<\frac{1}{\tan\alpha},
\qquad
\frac{\tan\beta-\beta}{\tan^2\beta}
<\frac{1}{\tan\beta}.
$$

Therefore,
$$
\begin{aligned}
&\left(\frac{\tan\alpha\tan\beta}
{\tan\alpha+\tan\beta}\right)^{3/2}
\left(
\frac{\tan\alpha-\alpha}{\tan^2\alpha}
+
\frac{\tan\beta-\beta}{\tan^2\beta}
\right)\\
&\qquad<
\left(\frac{\tan\alpha\tan\beta}
{\tan\alpha+\tan\beta}\right)^{3/2}
\left(
\frac{1}{\tan\alpha}
+
\frac{1}{\tan\beta}
\right)\\
&\qquad=
\left(\frac{\tan\alpha\tan\beta}
{\tan\alpha+\tan\beta}\right)^{1/2}.
\end{aligned}
$$

Since
$$
\frac{\tan\alpha\tan\beta}
{\tan\alpha+\tan\beta}
~<~\min \lbrace{\tan\alpha,\tan\beta}\rbrace
~=~\tan\alpha
~\leq~\sqrt3,
$$
we conclude that
$$
\begin{aligned}
\left(\frac{\tan\alpha\tan\beta}
{\tan\alpha+\tan\beta}\right)^{3/2}
\left(
\frac{\tan\alpha-\alpha}{\tan^2\alpha}
+
\frac{\tan\beta-\beta}{\tan^2\beta}
\right)
\,<\,3^{1/4}
<\frac{\sqrt2(1+\log4)}{\sqrt{\pi}}.
\end{aligned}
$$
Therefore
$$J(\Omega)~<~  {2(1+\log4)\over3\sqrt{\pi}}\,T^{3/2}~=~V(D),$$
where $D$ is a disc with the same area $\caL^2(D)=T$.   This completes the proof. 
\endproof

%

\section{Appendix: the slicing cost of a disc}
\label{sec:A}
\setcounter{equation}{0}
{\small
Let  $D$ be a disc of radius $R>0$.  To compute
its minimum slicing cost, we consider the slicing obtained by arcs of circumferences which cross 
the boundary $\partial D$
perpendicularly.    Notice that every such arc of circumference provides a solution to Dido's problem
\cite{Dido}.
Namely, it minimizes the length of the relative boundary among all subsets of $D$ with the same area.  As in Theorem~1.1 of \cite{BCM26}, this guarantees that the slicing strategy is optimal.

We start by parametrizing a half circumference of $\partial D$ by
$$\phi(x)~=~R-\sqrt{R^2-x^2}.$$
For every $x>0$ there is a circumference with center along the $y$-axis which crosses 
$\partial G$ perpendicularly at the point $\bigl(x, \phi(x)\bigr)$. Its center is at the point $(0, c(x))$,
$$c(x)~=~\phi(x)-x\,\phi'(x)~=~R-\sqrt{R^2-x^2}- {x^2\over \sqrt{R^2-x^2}}\,.$$
the radius is
\bel{rx}r(x)~=~x\sqrt{1+ [\phi'(x)]^2}~=~x\sqrt{1+{x^2\over R^2-x^2}}={x\over\sqrt{1-{x^2\over R^2}}}\,.
\eeq
\begin{figure}[ht]
\centerline{\hbox{\includegraphics[width=7cm]{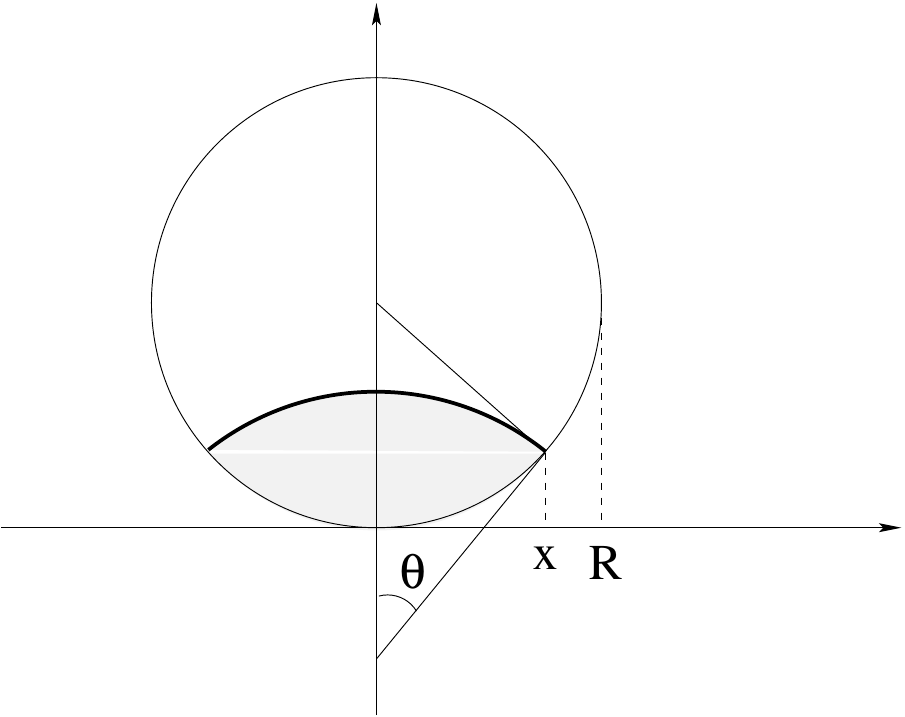}}}
\caption{Computing the slicing cost for a disc of radius $R$.}\label{f:sli7}
\end{figure}
From basic trigonometry, we recall that the area of the region between a circle of radius $r$ and a chord
of length $2l$ is 
$$A~=~\theta r^2 - l \sqrt{r^2-l^2}\,,\qquad\qquad \theta = \arcsin {l\over r}\,.$$

For every $x>0$, taking $l=x$, the area bounded by the two perpendicular circumferences 
with radii $R$ and $r(x)$ (see  Fig.~\ref{f:sli7})
is thus
\begin{align*}
A(x)~&=~R^2\arcsin{x\over R} - x\sqrt{R^2-x^2}+  
 r^2(x)\arcsin{x\over r(x)} - x \sqrt{r^2(x)-x^2}\\
 &=~R^2\arcsin{x\over R}-
 {xR\over\sqrt{1-{x^2\over R^2}}}+{x^2\over 1-{x^2\over R^2}}\arcsin\sqrt{1-{x^2\over R^2}}\,.
 \end{align*}
A slicing $t\to\Omega(t)$ should satisfy $\mathcal L^2\bigl(\Omega(t)\bigr)=t$, for all 
$0\leq t\leq T\doteq \pi R^2$. 
By symmetry, $J= J_1+J_2$, where
$$
J_2\,=\,J_1~\doteq~\int_0^{T/2}2\theta(x(A)) r(x(A))\, dA.$$
Here
$r(x)$ is the radius at (\ref{rx}) while the angle is
$$\theta(x)~=~\arcsin{x\over r(x)}~=~\arcsin\sqrt{1-{x^2\over R^2}}.$$
Since 
$$
A'(x)={2x\over\left(1-{x^2\over R^2}\right)^2}
\arcsin\sqrt{1-{x^2\over R^2}}-{2x^2\over R \left(1-{x^2\over R^2}\right)^{3/2}}\,,
$$
we obtain
\begin{align*}
J_1&~=~\int_0^{T/2}2\theta(x(A)) r(x(A))\, dA~=~\int_0^{R}2\theta(x) r(x)A'(x)\, dx\\
&=~4\int_0^R{x\over \sqrt{1-{x^2\over R^2}}}\arcsin\sqrt{1-{x^2\over R^2}}\left({x\over\left(1-{x^2\over R^2}\right)^2}
\arcsin\sqrt{1-{x^2\over R^2}}-{x^2\over R \left(1-{x^2\over R^2}\right)^{3/2}}\right)dx\\
&=~4R^3\int_0^R{1\over R\sqrt{1-{x^2\over R^2}}}
\left({{x^2\over R^2}\over \left(1-{x^2\over R^2}\right)^2}\left(\arcsin\sqrt{1-{x^2\over R^2}}\right)^2
-{{x^3\over R^3}\over\left(1-{x^2\over R^2}\right)^{3/2}}\arcsin\sqrt{1-{x^2\over R^2}}\right)dx\,.
\end{align*}
Applying the change of variable $\sigma=\arcsin\sqrt{1-{x^2\over R^2}}$, we obtain
$$
J_1~=~4 R^3\int_0^{\pi/2}\left(\sigma^2{(\cot \sigma)^2\over(\sin \sigma)^2}-\sigma(\cot \sigma)^3\right)d\sigma\,.
$$
Integrating by parts, one finds
\begin{align*}
\int&\left(\sigma^2{(\cot \sigma)^2\over(\sin \sigma)^2}-\sigma(\cot \sigma)^3\right)d\sigma~=~
\int \sigma^2{(\cot \sigma)^2\over(\sin \sigma)^2}d\sigma-{\sigma^2\over2}(\cot \sigma)^3-\int{3\over2} \sigma^2{(\cot \sigma)^2\over(\sin \sigma)^2}d\sigma\\
&\qquad\qquad=\,-{1\over2}\int\left(\sigma^2{(\cot \sigma)^2\over(\sin \sigma)^2}-\sigma(\cot \sigma)^3\right)d\sigma-{\sigma^2\over2}(\cot \sigma)^3
-{1\over2}\int \sigma(\cot \sigma)^3d\sigma,
\end{align*}
and hence
$$
\int\left(\sigma^2{(\cot \sigma)^2\over(\sin \sigma)^2}-\sigma(\cot \sigma)^3\right)d\sigma~=\,
-{\sigma^2\over3}(\cot \sigma)^3
-{1\over3}\int \sigma(\cot \sigma)^3d\sigma\,.
$$
Since
\begin{align*}
\int \sigma(\cot \sigma)^3d\sigma&=-{\sigma\over2(\sin \sigma)^2}-\sigma\log(\sin \sigma)+
\int\left({1\over2(\sin \sigma)^2}+\log(\sin \sigma)\right)d\sigma\\
&=-{\sigma\over2(\sin \sigma)^2}-\sigma\log(\sin \sigma)-{1\over2}\cot \sigma+
\int\log(\sin \sigma)d\sigma\,,
\end{align*}
one has
$$
\int\left(\sigma^2{(\cot \sigma)^2\over(\sin \sigma)^2}-\sigma(\cot \sigma)^3\right)d\sigma=
-{\sigma^2\over3}(\cot \sigma)^3+{\sigma\over6(\sin \sigma)^2}+{1\over6}\cot \sigma+
{1\over3}\sigma\log(\sin \sigma)-{1\over3}\int\log(\sin \sigma)d\sigma\,.$$
Finally, since $\sigma\log(\sin \sigma)\to0$  as $\sigma\to 0+$, and moreover
\begin{align*}
-{\sigma^2\over3}(\cot \sigma)^3+{\sigma\over6(\sin \sigma)^2}-{1\over6}\cot \sigma
&=~{1\over 6\sin \sigma^3}\big(-2\sigma^2(\cos \sigma)^3+\sigma\sin \sigma-(\sin \sigma)^2\cos \sigma\big)\\
&=~{-2\sigma^2+\sigma^2+\sigma^2+\O(1)\cdot\sigma^4\over 6\sigma^3}~=~\O(1)\cdot \sigma,
\end{align*}
we conclude
$$
J_1~=~4R^3\left({\pi\over12}-{1\over3}\int_0^{\pi/2}\log(\sin \sigma)d\sigma\right)=
{\pi R^3\over3}(1+\log4)\,.
$$
Here we have used the identity
 $$\int_0^{\pi/2}\log(\sin \sigma)d\sigma~=\,-{\pi\over2}\log2\,.$$
 By symmetry, recalling that $T\doteq \pi R^2$, the total cost of this slicing is
 $$J~=~2J_1
 ~=~{2\pi R^3\over3}(1+\log4)
~=~T^{3/2} {2\over3\sqrt{\pi}}(1+\log4)\,.$$

}

\end{document}